\UseRawInputEncoding
\documentclass[12pt, leqno]{article}
\usepackage{latexsym}
\usepackage{enumerate,amsmath,amsthm,amssymb,amsfonts,amscd}

\newtheorem{defn}{Definition}[section]
\newtheorem{thm}[defn]{Theorem}
\newtheorem{cor}[defn]{Corollary}
\newtheorem{prop}[defn]{Proposition}
\newtheorem{lem}[defn]{Lemma}

\numberwithin{equation}{section}

\def\F{\mbox{${\cal F}$}}

\def\O{\mbox{${\cal O}$}}

\newcommand{\cc}{\mathbb C}
\newcommand{\zz}{\mathbb Z}

\renewcommand{\a}{\alpha}

\input{arrow.tex}

\begin{document}
\title{\bf The functional analytic approach establishes the Corona theorem in the theory of SCV}
\author{S. R. PATEL}
\maketitle \vspace{1.5cm}
\noindent Department of Mathematics, C. U. Shah University, Wadhwan City,\\
Gujarat, INDIA\\
E-mail: srpatel.math@gmail.com\\
\begin{tabbing}
\> Secondary: 30H80, 32A38\\
\> 46E25, 46J10, 46J20
\end{tabbing}
\newpage
\noindent {\bf Abstract} We give an affirmative solution of the
{\it original} Corona problem (1962) for certain bounded
pseudoconvex domains or polydomains in $\cc^n$, including the
balls and polydiscs, by using the Gleason solvability property of
such domains for $H^\infty(G)$. In fact, we establish an
equivalency of the Corona problem and Gleason's problem (1964) in
the theory of several complex variables. Indeed, we extend our
recent works on both problems based on the functional analytic
approaches, and extend recent work of Clos, as well as apply the
usual Banach algebra method. As an application, we establish the
I. J. Schark's theorem for these $G$'s. We also obtain a welcome
extension of a result of Krantz and Li (1995) for a bounded,
strictly pseudoconvex domain in $\cc^n$ with $C^3$ boundary.
\newpage
\section{Introduction}
\noindent {\bf 1.1 History of the two problems}

\noindent {\bf 1.1.1 Corona problem} In 1942, Kakutani asked the
following fundamental question (known as {\it Corona problem}).
Let $H^\infty(D)$ be the Banach algebra of all bounded holomorphic
functions on the open unit disc $D$ in the complex plane $\cc$,
equipped with the usual sup-norm. Let $f_j \in H^\infty(D), \,j =
1, 2,\dots, n$, with the property that $\sum_{j = 1}^n|f_j(z)| >
\epsilon$ for some $\epsilon > 0$ and all $z \in D$ (namely, a
finite set of corona data on the disc $D$), do there exist $g_j
\in H^\infty(D), \,j = 1, 2,\dots, n$, such that $\sum_{j =
1}^nf_jg_j = 1$; i.e., $\sum_{j = 1}^nf_j(z)g_j(z) = 1$ for all $z
\in D$? This latter equation is known as the B\'{e}zout equation
~\cite[Theorem V.1.8]{20}.

The equivalent formulation of this question is to ask whether $D$
is dense in the maximal ideal space (aka spectrum)
$M(H^\infty(D))$ w.r.t. the Gel'fand (i.e., relative weak*-)
topology (equivalently, whether the corona is empty). Carleson
affirmatively solved the problem on the disc $D$ in 1962
~\cite{11}.

By a {\it domain} we mean an open, connected and bounded set $G$
in $\cc^n$. At this time there is no known domain in $\cc$ on
which this problem is failed.

Since 1962, the corona problem is open for arbitrary domains
(including the balls and polydiscs) in $\cc^n,\,n > 1$, and is a
lot more complicated. There have been investigations of this
problem on various types of domains, as well as in various Banach
spaces of holomorphic functions on these domains. For example, see
some partial results on the corona problem in ~\cite{9, 10, 32,
39}. Not only this, but a couple of counterexamples to the corona
problem in the theory of SCV have been produced in ~\cite{18, 42}.
The extensive exposition of history of the research of the problem
is outlined in ~\cite{16}; see \S 4 for details.

The main purpose of this paper is to provide favorable results on
the {\it original} corona problem for certain pseudoconvex/poly
domains $G$ (including the ball/polydisc) in $\cc^n$; see Theorems
2.4 and 3.1 and Corollary 2.5 below. As a consequence, we also
establish the weak corona theorem for $H^\infty(G)$ ~\cite{2, 3}.
In the process, we also obtain more general/extended results in
some cases (e.g., see ~\cite{30, 28}, and Theorem 2.6).

In fact, after the Carleson's solution, most of the subsequent
attempts to extend it for various Banach spaces of holomorphic
functions on more general domains in $\cc^n$ were aimed at solving
the B\'{e}zout's equation by various classical methods (mostly, by
solving the $\bar{\partial}-$equation); see, for instance,
~\cite{28, 30, 31, 44, 45}. These approaches required hard
classical analysis methods and has met serious difficulties in the
case when $n > 1$, {\bf not} leading to the solution to the
original $H^\infty$ corona problem; see, for instance, ~\cite{30,
31, 45}. For example, Varopoulos in ~\cite{45} realized that his
result {\it does not imply} that the corona problem fails in
several variables--only that the $\bar{\partial}-$technique with
that particular definition of Carleson measure fails (see \S 4).
We discuss a functional analytic approach different from those
classical analytic approaches.

In \S 2, we discuss an equivalency between the corona problem and
Gleason's problem (1964) in Theorems 2.4 and 2.7. This has become
possible as some results from ~\cite{37, 38} are the bases of the
idea of equivalency, as well as certain results of Clos, and
Kerzmann and Nagel are extended.

\noindent {\bf 1.1.2 Gleason's problem} Let $A(G)$ be the Banach
algebra of all holomorphic functions on a domain $G$ in $\cc^n$
that extend by continuity to the closure of $G$, equipped with the
sup-norm. Thus, it is a closed subalgebra of $H^\infty(G)$ of all
bounded holomorphic functions on $G$ such that the polynomials in
$z_1, z_2, \dots, z_n$ are dense in $A(G)$ (but not in
$H^\infty(G)$). For our convenience, we may consider $A(G)$ as the
Banach algebra of all uniformly continuous, holomorphic functions
on $G$. The Gleason's problem is to decide whether the ideal in
$A(G)$ (or $H^\infty(G)$), consisting of functions vanishing at a
fix point $\lambda = (\lambda_1, \lambda_2, \dots, \lambda_n) \in
G$, is algebraically finitely generated by the shifted coordinate
functions $z_j - \lambda_j$ for $j = 1, 2, \dots, n$ (see
~\cite{6} and other references therein for details). This is
actually a variant of the original Gleason's problem for $A(G)$,
$G$ the unit ball in $\cc^2$ and $\lambda = (0, 0)$ ~\cite{21}.

\noindent {\bf 1.2 Background} This paper contains continuation of
the works begun in ~\cite{37, 38}, and we shall feel free to use
the terminology and conventions established there. However, this
current work is specifically concerned with the solvability of
these two problems for $A(G)$ and $H^\infty(G)$, $G$ certain
bounded pseudoconvex/poly domains in $\cc^n$. In ~\cite{37}, we
extend the Gleason's result to finitely generated ideals in
Fr\'{e}chet algebras, answering affirmatively a question posed in
~\cite{12} for Fr\'{e}chet algebras. As consequences, locally
Stein algebras are completely characterized by intrinsic
properties within the class of Fr\'{e}chet algebras, using which
an affirmative answer to Gleason's problem for such algebras is
provided. In fact, we remark that complex analysts may find this
functional analytic approach interesting from applications point
of view. For example, see ~\cite[Theorems 4.2, 4.3 and Corollary
4.4]{37}, and references (and their reviews) to ~\cite{21} in
MathSciNet (MR0159241 (28 \#2458)). As we will see, this is,
indeed, the case.

Banach algebras satisfying ~\cite[Lemma 3.1]{37} (in particular,
the polydisc algebra $A(D^n)$ and the ball algebra $A(B^n)$) are
locally Stein-Banach algebras. In ~\cite{38}, we show that Banach
subalgebras $A^k(\overline{G})$ and $\textrm{Lip}^{k,\, \a}_H(G)$
of $H^\infty(G)$, for which we want to solve Gleason's problem,
are also locally Stein-Banach algebras. We now show that
$H^\infty(G)$ is a locally Stein-Banach algebra in Theorem 2.7.
Further, a subspace $X$ of $M(A)$ has the Gleason $A-$property if
the problem has an affirmative solution at all points of $X$ for a
locally Stein algebra $A$. Clos calls it the Gleason solvability
property of a domain $G$ for $H^\infty(G)$. The Gleason
solvability was first proved by Leibenson, albeit informally
~\cite{41}. The Gleason solvability for holomorphic Bergman
spaces, Bergman-Sobolev spaces and the holomorphic (harmonic)
Bloch spaces were studied in ~\cite{41, 35, 47} (~\cite{25}). The
domains considered were initially the unit ball, but then were
generalized to (bounded) strongly pseudo(convex) domains with
$C^2$ boundary.

In ~\cite[Corollary 4.4]{37}, we show that the dense open subspace
has the Gleason $A$-property, where $A$ is a semi-simple locally
Stein algebra. In particular, $A$ can be either $H^\infty(G)$ or
$A(G)$, where $G$ is a (strictly or weakly) pseudoconvex domain in
$\cc^n$ with various boundary conditions and containing the
origin. Thus the functional analytic method recaptures the
classical results on Gleason's problem obtained by Beatrous Jr.
~\cite{7}, Forn\ae ss and \O vrelid ~\cite{17}, Kerzman and Nagel
~\cite{27}, Noell ~\cite{34}, and Backlund and F\"{a}llstr\"{o}m
~\cite{6} (see other references therein for a list of papers on
this problem). The corona and Gleason's problems are recently
solved for Lipschitz algebras of holomorphic functions. In fact,
we affirmatively answer a question of Forn\ae ss and \O vrelid
(1983) for holomorphic H\"{o}lder and Lipschitz spaces ~\cite{38}.

Concerning the corona problem, as far as we know, almost all
attempts in the past were either relying on the solution of the
B\'{e}zout equation or relying on the use of the Koszul complex
(an idea developed for the computation of the cohomology of a Lie
algebra ~\cite{28, 44}) while applying the
$\bar{\partial}-$method. As we shall see, we find it necessary
(mathematically, ``sufficient") to consider the Gleason
solvability while solving the corona problem. On the other hand,
although we have not spelled it out explicitly, but have
considered the corona theorem as a hypothesis while solving
Gleason's problem ~\cite[Theorem 4.3 and Corollary 4.4]{37}. In
fact, we have a pleasant result of \.Zelazko ~\cite[Proposition
2]{46}, showing a possible connection between these two problems
through their algebraic formulations. Moreover Kerzman and Nagel
used the Koszul complex while solving Gleason's problem for
certain function algebras on certain domains in $\cc^n$. Now, all
this developments stem to think on an equivalency of these two
problems, as well as to give a short elegant proof, by using the
(maximal) ideal structures of Banach algebras $A(G)$ and
$H^\infty(G)$ alone. That means, we keep ourselves away from the
analytical aspects such as Gleason parts of the spectra of these
algebras.

In \S 2, one of our main goals is to extend ~\cite[Theorem 4]{13}.
In \S 3, we also solve the corona problem by applying a mixture of
our Banach algebra method and the most elegant method of Krantz
and Li ~\cite{30}. In \S 4, we give some concluding remarks on our
approach, leading to some open problems.
\section{Equivalency of the two problems}
Throughout the paper, ``algebra" will mean a complex, commutative
algebra with identity. As discussed above, we first show that
Theorem 6 of ~\cite{13} holds true for the Banach algebra $A(G)$,
where $G$ is a certain ``pseudoconvex" domain in $\cc^n$. This is
possible because of Theorem 5 of ~\cite{13} and the following
\begin{lem} \label{A(G)_subset} Let $G$ be a pseudoconvex domain
in $\cc^n$. Then $A(G)$ is a dense Banach subspace of the Bergman
space $B(G)$.
\end{lem}
\noindent {\bf Proof} Clearly, we have that $A(G) \subset
H^\infty(G) \subset \textrm{Hol}(G)$ w.r.t. the usual compact-open
topology. This implies that
$$A(G) \bigcap L^2(G)\subset H^\infty(G) \bigcap L^2(G) \subset
\textrm{Hol}(G)\bigcap L^2(G) = B(G),$$ where $B(G)$ is the
Bergman space of a pseudoconvex domain $G$. The polynominals in
$z_1,\,z_2,\,\dots,\,z_n$ are dense in $B(G)$ since $B(G)$ is
isometrically isomorphic to the weighted space
$\ell^\infty((\zz^+)^{n}, \frac{1}{|N| + 1})$ (see ~\cite{36} for
the analogous notion of the Beurling-Banach algebra
$\ell^1((\zz^+)^{n}, \omega)$ of weight type). Moreover, the
polynominals in $z_1,\,z_2,\,\dots,\,z_n$ are in $A(G)$. Since
each function of $A(G)$ is uniformly continuous on $G$, it extends
uniquely and in a norm preserving way to a continuous function on
$\overline{G}$ (see \S 1), it is obviously square-integrable on
$\overline{G}$. Hence $f\,\in\,L^2(G)$ whenever $f\,\in\,A(G)$,
and so, $f\,\in\,B(G)$. Thus $A(G)$ is a dense Banach subspace of
the Bergman space $B(G)$ (certainly, $A(G)$ is not a closed
subspace of $B(G)$ for an obvious reason). $\hfill \Box$

Thus we have the following crucial proposition, extending results
of ~\cite{13, 27}. Below, by a ``pseudoconvex" domain in $\cc^n$,
we mean a smooth strictly pseudoconvex domain in $\cc^n$
~\cite[Theorem 2]{23}, or a pseudoconvex domain with a Stein
neighbourhood basis ~\cite{40}; e.g., Stein domains in $\cc^n$.
\begin{prop} \label{Thm 4_A(G)} Let $G$ be a ``pseudoconvex" domain
in $\cc^n$. Suppose $f \in A(G)$ and $f(\a) = 0$ for some $\a \in
G$. Then there exists $f_1, f_2, \dots, f_n$ in $A(G)$ so that $f
= \sum_{j = 1}^n(z_j - \alpha_j)f_j$. In particular, $G$ is
Gleason solvable for $A(G)$.
\end{prop}
\noindent {\bf Proof} Since $f \in A(G)$, $f \in B(G)$ by Lemma
\ref{A(G)_subset}. Then, by ~\cite[Theorem 5]{13}, there exists
$f_1, f_2, \dots, f_n$ in $B(G)$ so that $f = \sum_{j = 1}^n(z_j -
\alpha_j)f_j$. Since $f \in A(G)$ and shifted coordinate functions
$z_j - \alpha_j$ for $j = 1, 2, \dots, n$, are in $A(G)$, it is
clear that $f_1, f_2, \dots, f_n$ are also in $A(G)$; $f_j =
T_j(f),\;j = 1, 2, \dots, n$, in Remarks A 2. below. There is also
a long route to show that $f_1, f_2, \dots, f_n$ are in
$A(G)$ by ~\cite[Theorem 3.2]{38}. $\hfill \Box$\\

\noindent {\bf Remarks A} 1. We recall that there are smooth
pseudoconvex domains in $\cc^n$ for which $L^\infty-$estimates for
$\bar{\partial}$ do not hold and so, uniform estimates for
$\bar{\partial}$ do not hold; see ~\cite{8, 18, 43}. Thus, the
above result is a significant generalization of ~\cite[Theorem
4]{13} for $A(G)$; also, the same is true for $H^\infty(G)$ in the
below theorem.

\noindent 2. The positive solution in the case of strictly
pseudoconvex domains with smooth boundaries was established in
~\cite{24}. It has been even strengthened in ~\cite{26} as
follows. The continuous linear operators
$$T_j : A(G) \rightarrow
A(G),\;f\,\mapsto\,T_j(f),\;j\,=\,1,\,2,\,\dots,\,n$$ solving the
Gleason's problem exist, which means that $$f(z) -
f(\lambda)\,=\,\sum_{j = 1}^n(z_j -
\lambda_j)(T_j(f))(z),\;z\,\in\,G, f\,\in\,A(G).$$ For polydomains
analogous decompositions are easily available by fixing all but
one variable method.

This implies, in particular, that the part of the spectrum
$M(A(G))$ in the fibers over $G$ is mapped in a bijective way to
$G$ by the Gel'fand transform
$(\hat{z}^1,\,\hat{z}^2,\,\dots,\,\hat{z}^n)$ of the coordinate
functions. By the result of ~\cite{23}, the entire spectrum
$M(A(G))$ is mapped bijectively onto $\overline{G}$ where $G$ is a
smooth pseudoconvex domain in $\cc^n$. It also suffices to
consider a pseudoconvex domain $G$ with a Stein neighbourhood
basis; see ~\cite{40}. Hence the Gel'fand topology coincides with
the Euclidean topology on $\overline{G}$ and $G$ is open in
$M(A(G))$.

As an application of Proposition 2.2, we extend Theorem 4 of
~\cite{13} for a ``polydomain" $G$ (including the polydisc in
$\cc^n$), obtained as a product of smooth strictly pseudoconvex
domains, in the following
\begin{thm} \label{Thm 5_A(G)} Let $G$ be either a ``pseudoconvex"
domain or else a ``polydomain" in $\cc^n$. Then $G$ is Gleason
solvable for $H^\infty(G)$.
\end{thm}
\noindent {\bf Proof} First, we recall that $A(G)$ is a closed
subalgebra of $H^\infty(G)$ w.r.t. the sup-norm such that the
polynomials in $z_1,\,z_2,\,\dots,\,z_n$ are dense in $A(G)$, but
not in $H^\infty(G)$. In fact, $\textrm{in} : A(G) \hookrightarrow
H^\infty(G)$ is a continuous injective homomorphism. So, the
adjoint spectral mapping $\textrm{in}^* : M(H^\infty(G))
\rightarrow M(A(G)) = \overline{G}$ is surjective and continuous
w.r.t. the Gel'fand topology. Indeed, its restriction is
one-to-one over $G$, and thus, the inverse mapping maps $G$
homeomorphically onto an open subset of $M(H^\infty(G))$. In fact,
we remark that, in ~\cite{1}, the one dimensional case of $G$
being the open unit disc, was explained, and claimed that the same
facts are probably true and can be proved at least whenever $G$ is
the open unit ball or the polydisc in $\cc^n$; for proof of the
latter claims, see ~\cite[Theorem 1]{33} and ~\cite{19}, which may
be applied to the case when $G$ is either a ``pseudoconvex" domain
or a ``polydomain" in $\cc^n$.

Next, we show that the maximal ideal $M_\alpha$, $\alpha \in G$,
of $H^\infty(G)$ is, indeed, algebraically finitely generated. For
this, we note that the maximal ideal $M_\alpha$, $\alpha \in G$,
of $A(G)$ is algebraically finitely generated by Proposition
~\ref{Thm 4_A(G)}. When $G$ is a polydomain in $\cc^n$, the
Gleason solvability is easily available through the analogous
decomposition, discussed in Remarks A 2., by fixing all but one
variable method. Then, by the discussion above, its inverse image,
denoted by $M_\alpha$ again, is also a maximal ideal of
$H^\infty(G)$. To show that it is, indeed, finitely generated, we
first show that the finitely generated maximal ideal $M_\alpha$ of
$A(G)$ is, in fact, an ideal of $H^\infty(G)$ as follows.

Let $f\,\in\,M_\a$ so that $f(\a)\,=\,0$ in $A(G)$. Let
$g\,\in\,H^\infty(G)$. To show that $fg\,\in\,M_\a$, we first show
that $fg$ extends continuously to $\overline{G}$. Since both
$A(G)$ and $H^\infty(G)$ are semi-simple Banach algebras, by the
standard Banach algebra theory,
$(\widehat{A(G)},\,\|\cdot\|_{\infty,\,A(G)})$ and
$(\widehat{H^\infty(G)},\,\|\cdot\|_{\infty,\,H^\infty(G)})$ are
also semi-simple Banach algebras. Also,
$M(\widehat{A(G)})\,\cong\,\overline{G}$ and
$M(\widehat{H^\infty(G)})\,\cong\,M(H^\infty(G))$. Moreover, their
respective Gel'fand transformations $f\,\rightarrow\,\hat{f}$ are
isometric isomorphism onto their Gel'fand representations. In
fact, for $A(G)$, the Gel'fand transformation is the identity
mapping such that
$\hat{f}(\phi_z)\,=\,\phi_z(f)\,=\,f(z),\;\textrm{for
all}\,z\,\in\,\overline{G}$. In the case of $H^\infty(G)$, the
restriction of the Gel'fand transformation is the identity
mapping. Since the maximal ideal $M_\alpha$, $\alpha \in G$, of
$A(G)$ is algebraically finitely generated by Proposition
~\ref{Thm 4_A(G)}, its inverse image $\widehat{M_\alpha}$ is also
an algebraically finitely generated maximal ideal of
$\widehat{A(G)}\,=\,A(G)$. Let
$\hat{g}\,\in\,\widehat{H^\infty(G)}$. It is enough to show that
$\hat{f}\hat{g}\,\in\,\widehat{A(G)}\,=\,A(G)$ (i.e.,
$\hat{f}\hat{g}$ extends continuously to $\overline{G}$), to claim
that the maximal ideal $\widehat{M_\alpha}$ of
$\widehat{A(G)}\,=\,A(G)$ is an ideal of $\widehat{H^\infty(G)}$.
This is obvious as
$\hat{g}\,:\,M(H^\infty(G))\,\supset\,\overline{G}\,\rightarrow\,\cc$,
and $\hat{g}$ is holomorphic on $G$ since the algebras
$\widehat{H^\infty(G)}\,\cong\,\widehat{H^\infty(G)}|_{\overline{G}}\,\cong\,\widehat{H^\infty(G)}|_G\,\cong\,H^\infty(G)$
algebraically. Thus, w.l.o.g., $M_\a$ of $A(G)$ is an ideal of
$H^\infty(G)$.

Let $\mathcal I$ be the family of finitely generated ideals of
$H^\infty(G)$ consisting of functions in $H^\infty(G)$ vanishing
at $\alpha$ in $G$. Since $M_\alpha \in \mathcal I$, $\mathcal
I\,\neq\,\emptyset$. Each ideal in $\mathcal I$ is generated by
the shifted co-ordinate functions $z_j -
\a_j,\,j\,=\,1,\,2,\,\dots,\,n,$ since $f(\a) = 0$. By the Zorn's
lemma, the maximal element of $\mathcal I$ is an algebraically
finitely generated maximal ideal. In fact, this maximal element is
not only a primary ideal, but is obviously contained in the
maximal ideal $M_\alpha$ of $H^\infty(G)$. From these facts, we
conclude that the maximal element is, indeed, the maximal ideal
$M_\alpha$ of $H^\infty(G)$. In particular,
$G$ is Gleason solvable for $H^\infty(G)$. $\hfill \Box$\\

\noindent {\bf Remarks B} 1. In ~\cite{14}, it was not evident
what relationship there was between the condition that $G$ be open
in $M(A)$ w.r.t. the Gel'fand topology, where $A$ is either
$H^\infty(G)$ or $A(G)$, and the condition that $G$ be an
$L^\infty$-pseudoconvex domain. However, Clos showed the one-way
implication (see ~\cite[Theorem 4]{13} and the fact, discussed in
~\cite{14}, that $G$ is open in $M(A)$ if it has the Gleason
$A$-property). One reason for interest in this discussion stems
from an application to the commuting Toeplitz operators problem on
Bergman spaces of pseudoconvex domains in $\cc^n$ (see ~\cite[p.
2]{13} for his comments).

\noindent 2. The proof of Theorem ~\ref{Thm 5_A(G)} is purely
algebraic. Since we use the corona theorem as a hypothesis in
Theorem 2.7 below, we obtain the Gleason solvability for a limited
class of domains in $\cc^n$ (see the below paragraph) than we
obtain the same for a larger class of domains in $\cc^n$ in
Theorem ~\ref{Thm 5_A(G)}. It is interesting to note that the
behavior of Gleason's problem on the boundary (and boundary
fibers) of a ``pseudoconvex" domain $G$ transitions from the
algebraic solvability to a study of the analytic and
analytic-geometric properties of $M(H^\infty(G))$. In particular,
the fibers over the boundary points of $G$ contain non-trivial
Gleason parts with analytic discs inside them (and a localized
version of Gleason's problem holds). However, since our focus is
to solve the B\'{e}zout's equation by using the Gleason
solvability property of $G$ in the below theorem, we do not need
to discuss the Gleason parts of $M(H^\infty(G))$ (see the second
last paragraph of Introduction).

\noindent 3. There are bounded domains in $\cc^n$ such that
neither they have a smooth boundary (e.g., Hartogs triangle) nor
they have a Stein neighbourhood basis (e.g., worm domain)
~\cite{15}. For the Hartogs triangle $T$, we have
$M(A(T))\,=\,\textrm{sp}(z_1,\,z_2,\,\dots,\,z_n)\,=\,
\overline{T}$ since the polynomials are dense in $A(T)$ and
$\overline{T}$ is polynomially convex. For the worm domain $W$, we
also have
$M(A(W))=\,\textrm{sp}(z_1,\,z_2,\,\dots,\,z_n)\,=\,\overline{W}$
by ~\cite[Theorem 2]{23}, or by ~\cite[Theorem 4.1]{43}. Thus the
corona theorem holds for $A(T)$ and $A(W)$, but it fails for
$H^\infty(T)$ and $H^\infty(W)$. We also recall that there are
smooth pseudoconvex domains $G$ in $\cc^n, \;n = 2, 3$, for which
the corona theorem fails for $H^\infty(G)$ ~\cite{18, 42}. These
domains are of type $EL^\infty$, i.e., they are not $H^\infty$
domain of holomorphy ~\cite[\S 4]{29}. However, for the same
domains $G$, the corona theorem holds for $A(G)$ by ~\cite[Theorem
4.1]{43}. This exhibits a significant difference between $A(G)$
and $H^\infty(G)$. That is, there exists a holomorphic function
$f$ in $A(G)$ which cannot be analytically continued to a larger
domain; i.e., these domains are {\bf not} of type $EA(G)$.
However, below, by a {\bf pseudoconvex} domain in $\cc^n$, we mean
a relatively compact, smooth strictly pseudoconvex domain in
$\cc^n$ such that its closure is an $S_\delta$ set (i.e., it has a
Stein neighbourhood basis) ~\cite[\S 5, Definition 2]{40}. Such a
pseudoconvex domain in $\cc^n$ possesses strictly pseudoconvex
boundary and is {\bf not} of type $HL^\infty$; i.e., it is a
$H^\infty$ domain of holomorphy ~\cite[Proposition 6]{29}.

In order to establish an equivalency between the corona problem
and Gleason's problem, we first show that the Gleason solvability,
obtained above, implies an affirmative answer to the corona
problem in the following
\begin{thm} \label{Corona} Let $G$ be either a {\bf pseudoconvex} domain or else a
``polydomain" in $\cc^n$. Then the corona problem holds true in
$H^\infty(G)$. That is, the spectrum $M(H^\infty(G))$ is the
Gel'fand closure of (the canonical image of) $G$.
\end{thm}
\noindent {\bf Proof} To show that the corona problem has an
affirmative solution for $G$, we use its equivalent formulation,
discussed in \S 1, namely, a finite {\it corona data} implies the
existence of the solution of the B\'{e}zout's equation. In fact,
we show that if there does not exist a solution of the
B\'{e}zout's equation, then there are no finite corona data
available. To this effect, we recall that there is a one-to-one
correspondence between the maximal ideals and the elements of the
spectrum of the Banach algebra $A$. As an easy consequence of this
fact, we further know that the algebraically finitely generated
ideal, generated by $f_1,\,f_2,\,\dots, f_n \in A$, is proper if
and only if
\begin{equation}\label{ArensBanach} \{\phi\,\in\,M(A)\,:\,\phi(f_j)\,=\,0,\;j\,=\,1,\,2,\,\dots,\,n\}\,\neq\,\emptyset.
\end{equation}
In our case, since $G$ is Gleason solvable by Theorem ~\ref{Thm
5_A(G)}, it turns out that $M_\alpha,\;\alpha \in G$, is
algebraically finitely generated maximal ideal if and only if (1)
holds if and only if there are no $g_1, g_2, \dots, g_n \in
H^\infty(G)$ such that $\sum_{j = 1}^nf_jg_j = 1$, where $f_j,\; j
= 1, 2, \dots, n,$ are the generators of $M_\alpha$. Thus there
does not exist a solution of the B\'{e}zout's equation if and only
if there are NO finite corona data available. That means, $\alpha
\in G$ is such that for every $\delta > 0$, $\sum_{j =
1}^n|f_j(\alpha)| < \delta$, which implies that $\sum_{j =
1}^n|f_j(\alpha)| = 0$ if and only if $f_j(\alpha) = 0,\; j = 1,
2, \dots, n$. Hence the corona problem holds true in
$H^\infty(G)$, where $G$ is either a {\bf pseudoconvex} domain or
else a ``polydomain" in $\cc^n$. $\hfill \Box$\\

\noindent {\bf Remarks C} Here are some remarks on our
Banach-algebraic technique in the above proof.

\noindent 1. In connection with Krantz's Theorem 1 and Remark 2
~\cite{28}, we note that if we apply the contrapositive method to
prove Theorem 1 of Krantz, then we must show that failure of the
corona theorem would imply that either there are NO corona data,
or the zeros of the corona data are NOT separated. However, it is
clear that there are NO corona data if and only if the zeros of
the functions in corona data are NOT separated, as explained in
the proof above. That means, one may drop the geometric
hypothesis, imposed in Theorem 1 of Krantz, in order to further
simplify that theorem. Now, if we apply the contrapositive method
to prove the simplified Theorem 1 of Krantz, then we have the
corona theorem (and that is what we have done in the above proof
for a finite corona data).

\noindent 2. Our Banach-algebraic technique in the above proof
also works for an infinite corona data (i.e., with infinitely many
generators $f\,=\,(f_1,\,f_2,\,\dots)$), taken appropriately (see
~\cite[\S 3.5]{16} for the disc in $\cc$) due to ~\cite[Brook's
Theorem, p. 136]{22}; note that we not only recover Theorem 5 of
Rosenblum and Tolokonnikov in some sense, but extend it for more
general $G$ in $\cc^n$.

\noindent 3. It is interesting to note that author very lately
(i.e., while finalizing the initial draft) came to know about the
{\bf ideal problem}, explained in the history of the corona
problem (see ~\cite[\S 7]{16}). Also, while splitting this initial
draft, author came to know about the method of proof of
~\cite[Lemma 2.3]{2}, establishing the I. J. Schark's theorem for
the Banach algebra $H^\infty(B_X)$ (or, its closed subalgebra
$A_u(B_X)$), where $B_X$ is the open unit ball of a Banach space
$X$. So our Banach-algebraic technique is purely based on some
well-known results from the Banach algebra theory as well as
author's recent work on Gleason's problem in ~\cite{37}. However,
some more comments are in order.

It is easy to see that our method here is similar to that of
Xiao's theorem on the multipliers of the Dirichlet space $\cal D$
in the unit disc (see ~\cite[p. 22]{16}) where it is discussed
that the corona data is necessary for the corona theorem to hold
(see also a remark on an alternative proof which also deals with
infinitely many generators; p. 23 of ~\cite{16}), and in Theorem
3.1 below, we show that the corona data alone is sufficient for
the corona theorem to hold (see ~\cite{30, 28}). Further, it is
easy to see that we have, in fact, solved the Rubel and Shields'
question for the function $\phi \,\equiv\,1$ in $H^\infty(G)$,
where $G$ is as above (including the ball and polydisc in
Corollary 2.5 below) ~\cite[\S 7, Questions 1, 2 and 3]{16}; for
Questions 2 and 3, we take $h\,\equiv\,1$ on $[0, \,\infty]$ and
$\phi\,\equiv\,1$ in the above theorem (note that, to answer
Question 3, we consider a maximal ideal $M_\a,\;\a\,\in\,G$,
generated by $f_1,\,f_2,\,\dots,\,f_n$ in the above theorem).

Since the ball and polydisc are (essentially the only) basic
examples of a pseudoconvex/poly domain in $\cc^n$ of our kind, we
have the following
\begin{cor} \label{Corona_Special} Let $G$ be either an open unit ball or a polydisc in
$\cc^n$. Then the corona problem holds true in $H^\infty(G)$.
\end{cor}
We also remark that if $G$ is a relatively compact, smooth
strictly pseudoconvex domain in a complex manifold $M$ such that
$\overline{G}$ is an $S_\delta$ set (equivalently, $G$ has a Stein
neighbourhood basis), then it possesses a strictly pseudoconvex
boundary. As commented in \S 1, our abstract (functional analytic)
method from ~\cite{37} has been applied to deduce the Gleason
$A$-property of a semi-simple, locally Stein algebra in Theorem
4.3, or Corollary 4.4, where $Y$ may be a domain (i.e., a
(reduced) Stein space) from the complex manifolds (of finite
dimension). Moreover, the Banach algebra method is also available
for a more general situation. Further, if $G$ is a relatively
compact domain in a Riemann surface $M$, then $M(A(G)) =
\overline{G}$, and so is true if $G$ is relatively compact domain
such that $\overline{G}$ is an $S_\delta$ set (equivalently, $G$
has a Stein neighbourhood basis) in a complex analytic manifold
$M$ (see ~\cite[Theorem 2.12]{40}). Hence, in view of a question
posed in ~\cite[4.1]{38}, we solve the corona problem for such
domains/(reduced) Stein spaces $G$ from the complex manifolds for
$H^\infty(G)$ in the following
\begin{thm} \label{corona_manifold} Let $G$ be a relatively
compact, smooth strictly pseudoconvex domain in a complex manifold
$M$ such that $\overline{G}$ is an $S_\delta$ set. Then the corona
problem holds true in $H^\infty(G)$.
\end{thm}
We have already shown in Theorem ~\ref{Thm 5_A(G)} that $G$ is
Gleason solvable for $H^\infty(G)$. However, as a part of our goal
to establish an equivalency between the two problems, in the
converse direction, we have
\begin{thm} \label{Gleason} Let $G$ be either a {\bf pseudoconvex} domain in
$\cc^n$ or else a ``polydomain" in $\cc^n$. Then $G$ is Gleason
solvable provided that the corona problem holds true in
$H^\infty(G)$.
\end{thm}
\noindent {\bf Proof} Let the corona problem hold true in
$H^\infty(G)$. Clearly, $H^\infty(G)$ is a semi-simple Banach
algebra of power series in $z_1, z_2, \dots, z_n$. By the
definition of locally Stein algebra ~\cite[Introduction]{37}, we
see that $H^\infty(G)$ is a locally Stein algebra since the
completion in the compact open topology of
$\widehat{H^\infty(G)}|_{G}$ is the Fr\'{e}chet algebra
$\textrm{Hol}(G)$. Alternatively, since $H^\infty(G)$ satisfies
the hypotheses of ~\cite[Lemma 3.1]{37}, it has an analytic
variety at $\phi_\a,\;\a \in G$, where $M_\a =
\textrm{ker}\,\phi_\a$. Hence $\widehat{H^\infty(G)}$ has an
analytic structure in each point $\phi_\a,\;\a \in G$. Now it is
easy to show that $H^\infty(G)$ is a locally Stein-Banach algebra
by applying proof of (ii) $\Rightarrow$ (i) of ~\cite[Theorem
4.1]{37}. Hence, by ~\cite[Theorem 4.3]{37} (applied to
$H^\infty(G)$), we see that Gleason's problem is solvable, since
$G$ is dense in
$M(H^\infty(G))$. $\hfill \Box$\\

\noindent{\bf Remark D} By Theorems 2.4 and 2.7, we see that the
two significant problems from the theory of SCV are equivalent.
Moreover the corona theorem implies the I. J. Schark's theorem
(see 4.1.1 below for details) for $H^\infty(G)$, where $G$ is as
in the above theorem (see, also, ~\cite{19, 33} for the same
result, with $G$ a strongly pseudoconvex domain in ~\cite{33}). In
fact, the corona problem has a topological nature (whether the
domain $G$ is dense in $M(H^\infty(G))$ w.r.t. the Gel'fand
topology) whereas Gleason's problem has only the algebraic nature
(whether a maximal ideal $M_\a,\; \a \in G$, is algebraically
finitely generated in $H^\infty(G)$; equivalently, whether $G$ has
the Gleason $H^\infty(G)-$property). The corona problem also has
an algebraic nature (given a finite corona data whether the
B\'{e}zout's equation holds true).
\section{Krantz and Li's approach for Corona problem.}
Krantz and Li restated the corona problem in the language of a
certain bounded operator which needs to be onto, if one wants to
solve the original problem ~\cite{30}. With this approach, they
succeed to solve the corona problem for Bergman spaces $A^p(G),
\;0 < p < \infty$, where $G$ is the unit ball as well as a bounded
strictly pseudoconvex domain in $\cc^n$ with $C^3$ boundary (but
for the range $1 < p < \infty$), for Hardy spaces $H^p(G), \;0\,<
p < \infty$ and for holomorphic Zygmund spaces $H\Lambda_\a(G),\;0
< \a < \infty$, where $G$ is a bounded strictly pseudoconvex
domain in $\cc^n$ with $C^3$ boundary (see Theorems 1.1 - 1.4 in
~\cite{30}). We remark that the corona data $f = (f_1, f_2, \dots,
f_m)$ belongs to $H^\infty(G)$ in Theorems 1.1 - 1.3, but it is
easy to see that they could have taken the corona data in the
respective Banach spaces as well since for $1 \leq p < \infty$,
$H^\infty \hookrightarrow H^p$ continuously. We also remark that
they could not solve the original corona problem for $H^\infty(G)$
(see ~\cite[Remark 1]{30}); also, they used the
$\overline{\partial}-$technique to prove these theorems, which is
not always possible for arbitrary strictly pseudoconvex domains
(see, also, ~\cite{45}). Indeed, we show that the bounded operator
$S_f$ is onto, solving the original corona problem in
\begin{thm} \label{KL_Sf} Let $G$ be a strictly pseudoconvex domain
in $\cc^n$ with $C^\infty$ boundary or a ``polydomain" in $\cc^n$.
Let $f\,=\,(f_1,\,f_2,\,\dots,\,f_m)$ be the corona data from the
product space $H^\infty(G)_m$. Then
$S_f\,:\,H^\infty(G)_m\,\rightarrow\,H^\infty(G)$ is onto. In
particular, the corona problem holds true in $H^\infty(G)$.
\end{thm}
\noindent {\bf Proof} This proof is, indeed, a little reflection
of the proof of Theorem 2.3 above. First, we show in the proof of
Theorem 2.7 that $H^\infty(G)$ is a semi-simple Banach algebra of
power series in $\F_n$, satisfying the hypotheses of ~\cite[Lemma
3.1]{37}. Now, following proof of (iii) of ~\cite[Lemma 3.1]{37},
we see that the product space $H^\infty(G)_m$ is, indeed, a Banach
algebra with an identity; the algebraic operations are
coordinatewise. Further, with a minor modification, we also see
that for each
$g\,=\,(g_1,\,g_2,\,\dots,\,g_m)\,\in\,H^\infty(G)_m$,
$$T_g\,:=\,\oplus_{i =
1}^mT_{g_i},\;\oplus_{i\,=\,1}^mh_i\,\mapsto\,\sum_{i\,=\,1}^m(f_i
- g_i)h_i,$$ is a bounded linear map on $H^\infty(G)_m$, where
$T_{g_i}\,:\,h\,\mapsto \,(f_i - g_i)h$ is a bounded linear
operator on $H^\infty(G)$. Now, given the corona data
$f\,=\,(f_1,\,f_2,\,\dots,\,f_m)$, let $I$ be the algebraically
finitely generated ideal in $H^\infty(G)$. Then $I$ is an improper
ideal if and only if $I$ cannot be contained in any maximal ideal
of $H^\infty(G)$, and $\sum_{i\,=\,1}^mf_ih_i\,=\,1$ for some
$h\,=\,(h_1,\,h_2,\,\dots,\,h_m)\,\in\,H^\infty(G)_m$. That is,
$S_f$ is onto (clearly, for the corona data
$f\,=\,(f_1,\,f_2,\,\dots,\,f_m)$, $S_f$ is obviously a bounded
linear map; in fact, $S_f \,=\,T_0$). $\hfill \Box$
\section{Concluding remarks on our approaches and open problems.}
Continuing our discussion from Introduction, in ~\cite{16}, it is
notable to recall Brudnyi's results related to the corona problem
for $H^\infty(R)$, where $R$ is a Caratheodory hyperbolic Riemann
surface ~\cite{10}. Recently, the corona problem has been solved
for some basic domains in $\cc^n$ such as $L^\infty-$pseudoconvex
domains ~\cite{44}. Recall that, for $n = 2$, this result was
obtained by Krantz ~\cite{28} (cf. Remarks C. 1.), and in
~\cite{39}, a similar conclusion is obtained under a much more
restrictive assumption that there exist $j \neq k$ such that
$$\overline{|f_j|^{-1}(-\epsilon,\,\epsilon)} \,\bigcap\,
\overline{|f_k|^{-1}(-\epsilon,\,\epsilon)} \,\bigcap\,
\partial\,G = \emptyset,\; \partial\,G \;\textrm{the
boundary of}\; G.$$ Also, Varopoulos noted between Theorems 3.1.1
and 3.1.2 why the Carleson's classical approach of solving
$\bar{\partial_b}-$problem fails in the SCV case while proving the
corona theorem (see Theorem on p. 272 of ~\cite{45} wherein the
solution functions are BMO).

Above, we solve the corona problem (algebraic formulation) by
considering the approach of solving the {\bf ideal problem} (which
is of algebraic nature). Thus this approach clearly shows the
possible connection between the corona problem and Gleason's
problem in the SCV case (cf. ~\cite[Proposition 2]{46}). However,
until we consider this connection for the first time in
~\cite[Theorem 4.3]{37}, as far as we know, no analysts did take
this route, and it is not a great surprise for this for obvious
reasons. Then, our extension of the results analogous to the
results of Clos also plays a significant role while filling the
gap via the Gleason solvability of certain ``pseudocconvex/poly"
domains $G$ when one wants to take the approach of solving the
ideal problem by using the Banach-algebraic technique. This
actually provides the necessary and sufficient conditions for the
ideal problem ~\cite[\S 7]{16}). Thus we have a very strong
connection between the two significant problems in the theory of
SCV through this equivalence. Moreover, the solvability of both
the problems for Lipschitz algebras of holomorphic functions in
~\cite{38} may provide us a concrete base for applying the
classical dead-end space method to get the solutions of these two
problems for various familiar Banach spaces of bounded holomorphic
functions (e.g., Hardy-Sobolev spaces, (holomorphic mean)
Besov-Lipschitz spaces and a weighted Bergman space) on bounded
symmetric domains in $\cc^n$ (including the unit ball).

\noindent {\bf 4.1 Open problems.} We now discuss a couple of open
problems in view of our approach as follows.

\noindent {\bf 4.1.1} It is worthwhile mentioning some work that
is somewhat (tangentially) related to the work given here (see
~\cite{1, 2, 3, 4, 5} for details). In fact, their work is
specifically concerned with the existence of (local) analytic
structure in spectra of certain algebras of bounded holomorphic
functions on the open unit ball $B_X$ of a complex Banach space
$X$. For example, they study the Gleason parts of spectra of these
algebras when $X \,=\, c_0$ in ~\cite{4}, and study the I. J.
Schark's theorem for Banach algebras $A_u(B_X)$ and
$H^\infty(B_X)$, where $X = c_0, \ell^1, \ell^2$, in ~\cite{2, 3},
and study the same theorem for the Fr\'{e}chet algebra $H_b(X)$,
where $X = \ell^1, \ell^2$, in ~\cite{3}. We establish this
theorem for holomorphic Lipschitz algebras in ~\cite{38}. To
establish a relation between their work with the current work, we
below conjecture that the question about the I. J. Schark's
theorem may have an affirmative solution.

In ~\cite[4.4]{37}, we have raised Gleason's problem (as well as
the Gleason solvability problem) in the infinite dimensional case
for Banach algebras $H^\infty(B_X)$ and ${\cal A}_u(B_X)$ (a
closed subalgebra of $H^\infty(B_X)$), where $B_X$ is the open
unit ball of a Banach space $X$. It is interesting to further
develop our approach (i.e., the abstract (functional analytic)
method and an equivalency of the two problems) so as to solve both
the problems for these algebras.

One reason for interest in solving the corona problem in the
infinite dimensional case stems from the point of view of the I.
J. Schark's Cluster Value Theorem in this case. The authors in
~\cite{2} concentrated only on obtaining several cluster value
theorems in this case, leaving to prove this theorem in the case
$H^\infty(B_X)$, where $X\,=\,\ell^2$ (see below).

Recall that if $f \in H^\infty(B_X)$ and $x \in
\overline{B_{X^{**}}}$, the cluster set of $f$ at $x$ is the set
$$\textrm{Cl}(f,\,x) :=  \{w \in \cc : \exists (x_n) \subset B_X
\,\textrm{that converges weak-star to}\; x,\,f(x_n) \rightarrow
w\}.$$ Then the I. J. Schark's {\it Cluster Value Theorem} is: For
$f$ and $x$ as above,
$$\textrm{Cl}(f,\,x)\,=\,\{\phi(f)\,:\,\pi(\phi)\,=\,x\}\,=\,\hat{f}(M_{x}),$$
where $\pi : M(H^\infty(B_X)) \rightarrow \cc$ and the fiber of
$M(H^\infty(B_X))$ over $x \in \overline{B_{X^{**}}}$ to be $M_x =
\pi^{-1}(x)$. In ~\cite{2}, this theorem is proved for $X = c_0$.
We conjecture that the theorem holds true for this case as well
(note that they could not prove it because the crucial lemma,
proved for the case $X = c_0$, fails in this case). This theorem
is also known as the ``poor man's/weak corona theorem", since if
the corona theorem holds true, then the I. J. Schark's theorem
also holds, but the converse may not hold. Thus, based on the
results obtained here, we now have a road map to establish the I.
J. Schark's cluster value theorems as follows. One first solves
the two questions from ~\cite[4.4]{37}, and then, as an
application, one establishes the corona problem, using the Banach
algebra method, which would also establish the I. J. Schark's
theorems.

\noindent {\bf 4.1.2} In connection to Theorems 2.4 and 2.7, it is
also natural to ask whether the I. J. Schark's theorem implies the
Gleason solvability for $H^\infty(G)$ and $A(G)$, where $G$ is as
in those theorems. The reason is: McDonald used the Gleason
solvability ~\cite[p. 215 (4)]{27} while proving the I. J.
Schark's theorem for $H^\infty(G)$, where $G$ is a smooth strongly
pseudoconvex domain in $\cc^n$.\\

\noindent {\bf Competing interests} The author has no relevant
financial or non-financial interests to disclose.


\begin{thebibliography}{Dillo 83}
\bibitem[1]{1} Aron, R.M., Cole, B., Gamelin, T.W.: Spectra of algebras of analytic functions on a Banach space. J. Reine Angew. Math. 415, 51-93 (1991)

\bibitem[2]{2} Aron, R.M., Carando, D., Gamelin, T.W., Lassale, S., Maestre, M.: Cluster values of analytic functions on a Banach space. Math. Ann. 353, 293-303 (2012)

\bibitem[3]{3} Aron, R.M., Carando, D., Lassale, S., Maestre, M.: Cluster values of holomorphic functions of bounded type. Trans. Amer. Math. Soc. 368, 2355-2369 (2016)

\bibitem[4]{4} Aron, R.M., Dimant, V., Lassale, S., Maestre, M.: Gleason parts for algebras of holomorphic functions in infinite dimensions. Rev. Mat. Complut. 33, 415-436 (2020)

\bibitem[5]{5} Aron, R.M., Falc\'{o}, J., Garc\'{i}a, D., Maestre, M.: Analytic structure in fibers. Studia Math. 240, 101-121 (2018)

\bibitem[6]{6} Backlund, U., F\"{a}llstr\"{o}m, A.: The Gleason property for Reinhardt domains. Math. Ann. 308, 85-91 (1997)

\bibitem[7]{7} Beatrous Jr, F.: H\"{o}lder estimates for the $\bar{\partial}$-equation with a support condition. Pacific J. Math. 90, 249-257 (1980)

\bibitem[8]{8} Berndtsson, B.: A smooth pseudoconvex domain in $\cc^2$ for which $L^\infty-$estimates for $\bar{\partial}$ do not hold. Ark. Mat. 31, 209-218 (1993)

\bibitem[9]{9} Berndtsson, B., Chang, S.Y., Lin, K.C.: Interpolating sequences in the polydisc. Trans. Amer. Math. Soc. 302, 161-169 (1987)

\bibitem[10]{10} Brudnyi, A.: Corona problem for $H^\infty$ on Riemann surfaces. In: R.G. Douglas, et al. (Eds.), The Corona Problem, Fields Institute Communications 72, pp. 31-45. Springer, New York (2014)

\bibitem[11]{11} Carleson, L.: Interpolation by bounded analytic functions and the corona problem. Ann. Math. (2) 76, 547-559 (1962)

\bibitem[12]{12} Carpenter, R.L.: Principal ideals in $F-$algebras. Pacific J. Math. 35, 559-563 (1970)

\bibitem[13]{13} Clos, T.G.: Solvability of the Gleason problem on a class of bounded pseudoconvex domains. Complex Anal. Oper. Theory 16, no. 58, 1-9 (2022)

\bibitem[14]{14} Clos, T.G., Izzo, A.J.: Approximation by an algebra generated by holomorphic and conjugate holomorphic functions. Rocky Mountain J. Math. 52, 1289-1294 (2022)

\bibitem[15]{15} Diederich, K., Forn\ae ss, J.E.: Pseudoconvex domains: An exmaple with non-trivial nebenh¨ulle. Math. Ann. 225, 275-292 (1977)

\bibitem[16]{16} Douglas, R.G., et al.: The Corona Problem, Connections Between Operator Theory, Function Theory, and Geometry. In: R.G. Douglas, et al. (Eds.), The Corona Problem, Fields Institute Communications 72, pp. 1-29. Springer, New York (2014)

\bibitem[17]{17} Forn\ae ss, J.E., \O vrelid, N.: Finitely generated ideals in $A(\Omega)$. Ann. Inst. Fourier (Grenoble) 33, 77-85 (1983)

\bibitem[18]{18} Forn\ae ss, J.E., Sibony, N.: Smooth pseudoconvex domains in $\cc^2$ for which the corona theorem and $L^p$ estimates for $\bar{\partial}$ fail. In: Complex Analysis and Geometry, Univ. Ser. Math., pp. 209-222. Plenum Press, New York (1993)

\bibitem[19]{19} Gamelin, T.W.: Iversen's theorem and fiber algebras. Pacific J. Math. 46, 389-414 (1973)

\bibitem[20]{20} Garnett, J.B.: Bounded Analytic Functions. Graduate Texts in Mathematics, Vol. 236, First edn. Springer, New York (2007)

\bibitem[21]{21} Gleason, A.M.: Finitely generated ideals in Banach algebras. J. Math. Mech. 13, 125-132 (1964)

\bibitem[22]{22} Goldmann, H.: Uniform Fr\'{e}chet algebras. North Holland Publ. Co., Elevier Science Publi., Amsterdam (1990)

\bibitem[23]{23} Hakim, M., Sibony, N.: Spectre de $A(\bar{\Omega})$ pour les domains born\'{e}s faiblement pseudoconvexes r\'{e}guliers. J. Funct. Anal. 37, 127-135 (1980)

\bibitem[24]{24} Henkin, G.M.: Approximation of functions in strictly pseudoconvex domains and a theorem of Z.L. Leibenzon. Bull. Acad. Polon. Sci S\'{e}r. Math. Astronom. Phys. 19, 37-42 (1971) (in Russian)

\bibitem[25]{25} Hu, Z.: Gleason's problem for harmonic mixed norm and Bloch spaces in convex domains. Math. Nachr. 279, 164-178 (2006)

\bibitem[26]{26} Jak\'{o}bczak, P.: Extension and decomposition operators in products of strictly pseudoconvex sets. Annales Polon. Math. 44, 219-237 (1984)

\bibitem[27]{27} Kerzman, N., Nagel, A.: Finitely generated ideals in certain function algebras. J. Funct. Anal. 7, 212-215 (1971)

\bibitem[28]{28} Krantz, S.G.: The corona problem with two pieces of data. Proc. Amer. Math. Soc. 138, 3651-3655 (2010)

\bibitem[29]{29} Krantz, S.G.: The corona problem in several complex variables. In: R.G. Douglas, et al. (Eds.), The Corona Problem, Fields Institute Communications 72, pp. 107-126. Springer, New York (2014)

\bibitem[30]{30} Krantz, S.G., Li, S.Y.: Some remarks on the corona problem on strongly pseudoconvex domains in $\cc^n$. Illinois J. Math. 39, 323-349 (1995)

\bibitem[31]{31} Krantz, S.G., Li, S.Y.: Explicit solutions for the corona problem with Lipschitz data in the polydisc. Pacific J. Math. 174, 443-458 (1996)

\bibitem[32]{32} Krantz, S.G., Li, S.Y.: Factorization of functions in subspaces of $L^1$ and applications to the corona problem. Indiana Univ. Math. J. 45, 83-102 (1996)

\bibitem[33]{33} McDonald, G.: The maximal ideal space of $H^\infty + C$ on the ball in $\cc^n$. Canad. J. Math. 31, 79-86 (1979)

\bibitem[34]{34} Noell, A.: The Gleason problem for domains of finite type. Complex Variables Theory Appl. 4, 233-241 (1985)

\bibitem[35]{35} Ortega, J.M.: The Gleason problem in Bergman-Sobolev spaces. Complex Variables Theory Appl. 20, 157-170 (1992)

\bibitem[36]{36} Patel, S.R.: Uniqueness of the Fr\'{e}chet algebra topology on certain Fr\'{e}chet algebras. Studia Math. 234, 31-47 (2016)

\bibitem[37]{37} Patel, S.R.: On functional analytic approach for Gleason's problem in the theory of SCV, Banach J. Math. Anal. 16, no. 47, 1-17 (2022)

\bibitem[38]{38} Patel, S.R.: On functional analytic approach for corona and Gleason's problems for holomorphic Lipschitz algebras. Extracta Math. 41, 159-184 (2026)

\bibitem[39]{39} Preda, O.: The corona problem with restrictions on the relative position of sublevel sets. Arch. Math. (Basel) 105, 563-569 (2015)

\bibitem[40]{40} Rossi, H.: Holomorphically convex sets in several complex variables. Ann. Math. 74, 470-493 (1961)

\bibitem[41]{41} Rudin, W.: Function theory in the unit ball of $\cc^n$. Classics in Mathematics, Springer-Verlag, Berlin (2008) Reprint of the 1980 edition

\bibitem[42]{42} Sibony, N.: Probl\'{e}me de la couronne pour des domaines pseudoconvexes \'{a} bord lisse. Ann. Math. (2) 126, 675-682 (1987)

\bibitem[43]{43} Sibony, N.: Some aspects of weakly pseudoconvex domains. Proc. Symp. Pure Math. 52, 199-231 (1991)

\bibitem[44]{44} Tikaradze, A.: On the corona problem for strongly pseudoconvex domains. Analysis Math. 48, 1209-1212 (2022)

\bibitem[45]{45} Varopoulos, N.Th.: BMO functions and the $\bar{\partial}-$equation. Pacific J. Math. 71, 221-273 (1977)

\bibitem[46]{46} \.Zelazko, W.: On maximal ideals in commutative m-convex algebras. Studia Math. 58, 291-298 (1976)

\bibitem[47]{47} Zhu, K.H.: The Bergman spaces, the Bloch space, and Gleason's problem. Trans. Amer. Math. Soc. 309, 253-268 (1988)
\end{thebibliography}
\end{document}